\documentclass[12pt]{amsproc}
\usepackage{amssymb}
\usepackage{amsmath}
\usepackage{amsfonts}
\usepackage{geometry}
\usepackage{graphicx}%
\providecommand{\U}[1]{\protect\rule{.1in}{.1in}}
\theoremstyle{plain}

\numberwithin{equation}{section}

\newcommand{\lap}{\mbox{$\bigtriangleup$}}
\newcommand{\grad}{\mbox{$\bigtriangledown$}}
\newtheorem{mthm}{Theorem}

\newtheorem{lem}{Lemma}[section]

\newtheorem{rem}{Remark}[section]

\begin{document}
\title[An overdetermined problem in Riesz-potential]{An overdetermined problem in Riesz-potential and fractional Laplacian}
\author{Guozhen Lu}
\author{Jiuyi Zhu}
\address{Guozhen Lu and Jiuyi Zhu\\
Department of Mathematics\\
Wayne State University\\
Detroit, MI 48202, USA\\
Emails: gzlu@math.wayne.edu and jiuyi.zhu@wayne.edu  }
\thanks{Research is partly supported by a US NSF grant \#DMS0901761.}
\date{}
\subjclass{31B10, 35N25} \keywords{Overdetermined problem, Riesz
potential, moving plane method in integral form, fractional
Laplacian.} \dedicatory{ }

 \begin{abstract}
The main purpose of this paper is to address two open questions
raised by W. Reichel in \cite{R2} on characterizations of balls in
terms of the Riesz potential and fractional Laplacian. For a bounded
$C^1$ domain $\Omega\subset \mathbb R^N$, we consider the
Riesz-potential
$$u(x)=\int_{\Omega}\frac{1}{|x-y|^{N-\alpha}} \,dy$$ for $2\leq
\alpha \not =N$. We show that $u=$ constant on $\partial \Omega$ if
and only if $\Omega$ is a ball. In the case of $ \alpha=N$, the
similar characterization is established for the logarithmic
potential $u(x)=\int_\Omega \log{\frac{1}{|x-y|}} \,dy$. We also
prove that such a characterization holds for the logarithmic Riesz potential $$u(x)=\int_\Omega
|x-y|^{\alpha-N}\log\frac{1}{|x-y|} \,dy$$ when the diameter of the
domain $\Omega$ is less than $e^{\frac{1}{N-\alpha}}$ in the case
when $\alpha-N$ is a nonnegative even  integer. This provides a
characterization for the overdetermined problem of the fractional
Laplacian. These results  answer two open questions in \cite{R2} to
some extent.

\end{abstract}
\maketitle

\section{Introduction}

It is well-known that the gravitational potential of a ball of
constant mass density is constant on the surface of the ball.
 It is shown by Fraenkel \cite{Fr} that this property indeed provides a characterization of balls.
In fact, Fraenkel proves the following

{\bf Theorem A \cite{Fr}}: Let $\Omega\subset \mathbb R^N$ be a
bounded domain and $\omega_N$ be the surface measure of the unit
sphere in $\mathbb R^N$. Consider
\begin{equation}
u(x)=\left \{ \begin{array}{lll} \frac{1}{2 \pi} \int_\Omega \log{\frac{1}{|x-y|}} \,dy, \qquad \qquad  & N=2,\\
\\
\frac{1}{(N-2)\omega_N} \int_\Omega \frac{1}{|x-y|^{N-2}} \,dy,
\qquad \qquad & N\geq 3.
\end{array} \right.
\label{Fr}
\end{equation}
If $u(x)$ is constant on $\partial \Omega$, then $\Omega$ is a ball.

This result has been extended by Reichel \cite{R2} to more general
Riesz potential, but under a more restrictive assumption on the
domain $\Omega$, i.e., $\Omega$ is assumed to be convex. In
\cite{R2},  Reichel considers the integral equation
\begin{equation}
u(x)=\left \{ \begin{array}{lll} \int_\Omega \log{\frac{1}{|x-y|}} \,dy, \qquad \qquad  & N=\alpha,\\
\\
\int_\Omega \frac{1}{|x-y|^{N-\alpha}} \,dy, \qquad \qquad & N\neq
\alpha,
\end{array} \right.
\label{int}
\end{equation}
and proves the following theorem.

{\bf Theorem B \cite{R2} }: Let $\Omega\subset \mathbb R^N$ be a
bounded convex domain and $\alpha >2$, if $u(x)$ is constant on
$\partial \Omega$, then $\Omega$ is a ball.

This more general Riesz potential is actually closely related to the
fractional Laplacian $(-\lap)^{\frac{\alpha}{2}}$ in $\mathbb R^N$.
Let $\mathbb{N}_0$ be the collection of nonnegative integers. It is
known that the fundamental solution $G(x,y)$ for pseudo-differential
operator $(-\lap)^{\frac{\alpha}{2}}$ in $\mathbb R^N$ has the
following representation
\begin{equation}
G(x,y)=\left\{
\begin{array}{ll}
\frac{\Gamma(\frac{N-\alpha}{2})}{2^\alpha \pi^{\frac{N}{2}}
\Gamma(\frac{\alpha}{2})}|x-y|^{\alpha-N},  \qquad \qquad
&\mbox{if}\ \ \frac{\alpha-N}{2}\not\in \mathbb{N}_0,\\
\\
\frac{(-1)^k}{2^{\alpha-1}\pi^{\frac{N}{2}}\Gamma({\frac{\alpha}{2}})}|x-y|^{\alpha-N}\log\frac{1}{|x-y|},
\qquad \qquad &\mbox{if}\ \ \frac{\alpha-N}{2}\in \mathbb{N}_0.
\end{array} \right.
\label{green}
\end{equation}

 We note that for the case of $\alpha=2$,  Fraenkel's result is under weaker assumption on the domain $\Omega$, namely, $\Omega$ only needs to be bounded and open in $\mathbb R^N$.
The surprising part for $\alpha=2$ is that there is neither
regularity nor convexity requirement for $\Omega$.  Thus, two open
problems were  raised by Reichel in \cite{R2}

{\bf Question 1.}  Is  Theorem B  true if we remove the convexity
assumption of $\Omega$?
\medskip

{\bf Question 2.} Is there an analogous result as Theorem B for
Riesz-Potential of the form
  \begin{equation}
  u(x)=\int_\Omega |x-y|^{\alpha-N}\log\frac{1}{|x-y|} \,dy?
  \label{log}
  \end{equation}\\
It is meaningful to study (\ref{log}) because in the case of
$\frac{\alpha-N}{2}\in \mathbb N_0$, up to some rescaling, the
kernel function in above integral is the fundamental solution of the
fractional Laplacian $(-\lap)^{\frac{\alpha}{2}}.$

Our goal is to address the above two open questions.

The first result we establish does remove the convexity assumption
in Theorem B.

\begin{mthm}
Let $\Omega$ ba a $C^1$ bounded domain. If $u$ in (\ref{int}) is
constant on $\partial\Omega$, then $\Omega$ is a ball. \label{1t}
\end{mthm}

 As far as Question 2 is concerned, we partially solve it under some additional assumption on the diameter of the domain $\Omega$. Since we are only interested in the case when $\alpha>N$, we will assume this when we address Question 2.

\begin{mthm}
Assume $\alpha>N$. Let $\Omega$ be a $C^1$ bounded domain with
$diam \,\Omega< e^{\frac{1}{N-\alpha}}$. Thus, $\Omega$ is a ball if
$u(x)$ in (\ref{log}) is constant on $\partial\Omega$. \label{2t}
\end{mthm}

\begin{rem}
In the above two theorems, if the conclusion that $\Omega$ is a ball
is verified, then we can easily deduce that $u(x)$ is radially
symmetric with respect to the center of the ball.
\end{rem}

There has been extensive study in the literature about
overdetermined problems in elliptic differential equations and
integral equations. In his seminal paper \cite{Se}, Serrin showed
that the overdetermined boundary value determines the geometry of
the underlying set. This is, if $\Omega$ is a bounded $C^2$ domain
and $u\in C^2(\bar\Omega)$ satisfies the following
\begin{equation}
\left \{ \begin{array}{ll} \lap u=-1    & \mbox{in} \ \ \Omega, \\
\\
u=0, \qquad \ \frac{\partial u}{\partial n}=constant &\mbox{on}\ \
\Omega,
\end{array}
\right.
\end{equation}
then $\Omega$ is a ball and $u$ is radially symmetric with respect
to its center of the ball. Serrin's proof  is based on what is
nowadays called the moving planes method relying on the maximum
principle of solutions to the differential equations,  which is
originally due to Alexandrov, and has been later used to derive
further symmetry results for more general elliptic equations.
  Important progress as for the moving plane methods since then are the works of Gidas-Ni-Nirenberg \cite{GNN},
   Caffarelli-Gidas-Spruck \cite{CGS},  to just name some of the early works in this direction.

Immediately after Serrin's paper, Weinberger \cite{W} obtained a
very short proof of the same result, using the maximum principle
applied to an auxiliary function. However, compared to Serrin's
approach, Weinberger's proof  relies crucially on the linearity of
the Laplace operator.

Since the work of \cite{Se}, many results are obtained about
overdetermined problems. The interested reader may refer to
\cite{AB}, \cite{B}, \cite{BK},  \cite{BNST}, \cite{BNST1}, \cite{CS},
\cite{EP}, \cite{FG}, \cite{FGK}, \cite{FK}, \cite{FV}, \cite{G},
\cite{GL}, \cite{HPP}, \cite{Lim}, \cite{Liu}, \cite{M}, \cite{MR},
\cite{PP}, \cite{PS}, \cite{P}, \cite{Sh}, \cite{Si},  \cite{WX}
and references therein, for more general elliptic equations.  See
also \cite{R1} and reference therein for overdetermined problems in
an exterior domain or general domain. In \cite{BNST}, an alternative
shorter proof of Serrin's result, not relying explicitly on the
maximum principle has been  given, where they deduce some global
information concerning the geometry of the solution.

Overdetermined problems are important from the point of view of
mathematical physics. Many models in fluid mechanics, solid
mechanics, thermodynamics, and electrostatics are relevant to the
overdetermined Dirichlet or Newmann boundary problems of elliptic
partial differential equations.
 We refer the reader to the article \cite{FG} for a nice introduction in that aspect.

Instead of a volume potential, single layer potential is also
considered in overdetermined problems. A single layer potential is
given by
\begin{equation}
u(x)=\left \{ \begin{array}{lll}
 A \int_{\partial\Omega}\frac{-1}{2 \pi} \log{\frac{1}{|x-y|}} \,d\sigma_y, \qquad \qquad  & N=2,\\
\\
A \int_{\partial\Omega}
\frac{1}{(N-2)\omega_N}\frac{1}{|x-y|^{N-2}}\,d\sigma_y, \qquad
\qquad & N\geq 3,
\end{array} \right.
\label{sig}
\end{equation}
where $A>0$ is the constant source density on the boundary of the
domain $\Omega$. If $u$ is constant in $\bar\Omega$, then $\Omega$
can be proved to be a ball under different smoothness assumption  on
the  domain $\Omega$.
 See \cite{M} for the case of $n=2$ and
\cite{R1} for the case of $n\geq 3$, and also some related works in
\cite{Lim} and \cite{Sh}. We also refer the reader to the book of C.
Kenig \cite{K} on this subject of layer potential.

Generally speaking,  two approaches are widely applied in dealing
with overdetermined problems. One is the classical moving plane
method. In \cite{Se}, the moving plane method with a sophisticated
version of Hopf boundary maximum principle plays a very important
role in the proof. The other way is based on an equality of Rellich
type, as well as an interior maximum principle, see \cite{W}. Our
approach is a new variant of moving plane method - Moving plane in
integral forms. It is much different from the traditional methods of
moving planes used for partial differential equations. Instead of
relying on the differentiability and maximum principles of the
structure, a global integral norm is estimated. The method of moving
planes in integral forms can be adapted to obtain symmetry and
monotonicity for solutions. The method of moving planes on integral
equations was  developed in the work of  W. Chen, C. Li and B. Ou
\cite{CLO}, see also  Y.Y. Li \cite{Li}, the book by W. Chen and C.
Li \cite{CL1} and an exhaustive list of  references therein, where
the symmetry of solutions in the entire space was proved. Moving
plane method in integral form over bounded domains requires some additional 
efforts and has  been  carried out recently in symmetry
problems arising from the integral equations over bounded domains,
see the work of D. Li, G. Strohmer and L. Wang \cite{LSW}.

We end this introduction with the following remark concerning the
characterization of balls by using the Bessel potential. The Bessel
kernel $g_\alpha$ in $\mathbb R^N$ with $\alpha\geq 0$ is defined by
\begin{equation}
g_\alpha(x)=\frac{1}{r(\alpha)}\int_0^{\infty}
\exp(-\frac{\pi}{\delta}|x|^2)\exp(-\frac{\delta}{4\pi})\delta^{\frac{\alpha-N-2}{2}}
\,d\delta, \label{ex1}
\end{equation}
where $r(\alpha)=(4\pi)^{\frac{\alpha}{2}}\Gamma(\frac{\alpha}{2}).$\\

In  the paper \cite{HLZ}, we consider the Bessel potential type
equation:
\begin{equation}
u(x)=\int_\Omega g_\alpha(x-y) \,dy. \label{b1}
\end{equation}

Overdetermined problems for Bessel potential over a bounded domain in $\mathbb R^N$ have
been recently studied in \cite{HLZ}. For instance,
 the following theorem is proved in \cite{HLZ}, among some other results:
\begin{mthm}
Let $\Omega$ ba a $C^1$ bounded domain in $\mathbb R^N$. If $u$ in
(\ref{b1}) is constant on $\partial\Omega$, then $\Omega$ is a ball.
\label{4t}
\end{mthm}

It is well-known that (\ref{b1}) is closely related to the following
fractional equation
$$(I-\lap)^{\frac{\alpha}{2}}u=\chi_\Omega.$$
In the case of $\alpha=2$, it turns out to be the ground state of
the Schr$\ddot{o}$dinger equation.

The paper is organized as follows. In Section 2, we show Theorem 1.
In Section 3, we carry out the proof of Theorem 2.  Throughout this
paper, the positive constant $C$ is frequently used in the paper. It
may differ from line to line,  even within the same line. It also
may depends on $u$ in some cases. 

Finally, we thank Dr. Xiaotao Huang for his comments on our earlier draft of this paper.

\section{Proof of Theorem \ref{1t}}

In this section, we will prove Theorem \ref{1t} by adapting the
moving plane method in integral forms, see \cite{CLO}. Since we are
dealing with the case of bounded domains, we modify the method
accordingly (see also \cite{LSW}, \cite{CZ}).

We first introduce some notations. Choose any direction and, rotate
coordinate system if it is necessary such that $x_1$-axis is
parallel to it. For any $\lambda\in \mathbb R$, define
$$T_\lambda=\{(x_1,...,x_n)\in \Omega | x_1=\lambda \}.$$

Since $\Omega$ is bounded, if $\lambda$ is sufficiently negative,
the intersection of $T_\lambda$ and $\Omega$ is empty. Then, we move
the plane $T_\lambda$ all the way to the right until it intersects
$\Omega$. Let
$$\lambda_0=min\{\lambda :T_\lambda\cap\bar\Omega\not= \emptyset\}.$$
For $\lambda>\lambda_0$, $T_\lambda$ cuts off $\Omega$. We define
$$ \Sigma_\lambda=\{x\in \Omega |x_1<\lambda\}.$$
Set
$$x_\lambda=\{2\lambda-x_1,...,x_n\} $$
and
$$
\Sigma^\prime_\lambda=\{x_\lambda\in\Omega|x\in\Sigma_\lambda\}.$$

At the beginning of $\lambda>\lambda_0$, $ \Sigma^\prime_\lambda$
remains within $\Omega$. As the plane keeps moving to the right,  $
\Sigma^\prime_\lambda$ will still stay in $\Omega$ until at least
one of the following events occurs:

(i)$ \Sigma^\prime_\lambda$ is internally tangent to the boundary of
$\Omega$ at some point $P_\lambda$ not on $T_\lambda$.
\medskip

(ii) $T_\lambda$ reaches a position where it is orthogonal to the
boundary of $\Omega$ at some point $Q$.\\ Let $\bar\lambda$ be the
first value such that at least one of the above positions is
reached.

We assert that $\Omega$ must be symmetric about $T_{\bar\lambda}$;
i.e.,
\begin{equation}
\Sigma_{\bar\lambda}\cup T_{\bar\lambda}\cup
\Sigma^\prime_{\bar\lambda}=\Omega. \label{key}
\end{equation}
If this assertion is verified, for any given direction in $\mathbb
R^N$, there also exists a plane $T_{\bar\lambda}$ such that $\Omega$
is symmetric about $T_{\bar\lambda}$. Moreover, $\Omega$ is
connected. Then the only domain with those properties is a ball, see
\cite{Al}.

In order to assert (\ref{key}), we introduce
$$u_\lambda(x)=u(x_\lambda),$$
$$
\Omega_\lambda=\Omega\backslash(\overline{\Sigma_\lambda\cup\Sigma^\prime_\lambda}).$$

 We first establish  some lemmas. Throughout the paper we assume
 $\alpha\geq 2$.

 \begin{lem}
 Let $l\in\mathbb N$ with $1\leq l<\alpha$. Then for any solution in
 (\ref{int}), $u\in C^l(\mathbb R^N)$ and differentiation of order $l$
 can be taken under the integral.
\label{mono}
 \end{lem}
 \begin{proof}
 The proof is standard. We refer the reader  to \cite{R2}.
 \end{proof}

 \begin{lem}
For
 $\lambda_0<\lambda<\bar\lambda$ and $u(x)$ satisfying (\ref{int}), we have

(i) If $N\geq \alpha$, $u_\lambda(x)>u(x)$ for any $x\in
\Sigma_\lambda.$
\medskip

(ii) If $N< \alpha$, $u_\lambda(x)<u(x)$ for any $x\in
\Sigma_\lambda.$

\end{lem}

\begin{proof}
 For $x\in\Sigma_\lambda$, in the case of $N=\alpha$, we rewrite
$u(x)$ and $u_\lambda(x)$ as

$$u(x)=\int_{\Sigma_\lambda} \log\frac{1}{|x-y|} \,dy + \int_{\Sigma_\lambda} \log\frac{1}{|x_\lambda-y|} \,dy
+\int_{\Omega_\lambda} \log\frac{1}{|x-y|} \,dy,$$ and

$$u_\lambda(x)=\int_{\Sigma_\lambda} \log\frac{1}{|x_\lambda-y|} \,dy + \int_{\Sigma_\lambda} \log\frac{1}{|x-y|} \,dy
+\int_{\Omega_\lambda} \log\frac{1}{|x_\lambda-y|} \,dy.$$ Then
\begin{equation}
u_\lambda(x)-u(x)=\int_{\Omega_\lambda}\log\frac{|x-y|}{|x_\lambda-y|}
\,dy. \label{Com}
\end{equation}
 Since $|x-y|>|x_\lambda-y|$ for $x\in \Sigma_\lambda$ and
$y\in \Omega_\lambda$, then
 $$u_\lambda(x)>u(x).$$
 While in the case of $ N\not =\alpha$, $u_\lambda(x)$ and $u(x)$ have the
 following representations respectively:
 $$u(x)=\int_{\Sigma_\lambda}|x-y|^{\alpha-N}\,dy + \int_{\Sigma_\lambda} |x_\lambda-y|^{\alpha-N} \,dy
+\int_{\Omega_\lambda} |x-y|^{\alpha-N} \,dy,$$
 and
 $$u_\lambda(x)=\int_{\Sigma_\lambda}|x_\lambda-y|^{\alpha-N}\,dy + \int_{\Sigma_\lambda} |x-y|^{\alpha-N} \,dy
+\int_{\Omega_\lambda} |x_\lambda-y|^{\alpha-N} \,dy.$$ Thus,
\begin{equation}
u_\lambda(x)-u(x)=\int_{\Omega_\lambda}(|x_\lambda-y|^{\alpha-N}-|x-y|^{\alpha-N})
\,dy, \label{Com1}
\end{equation}
 Note that  $|x-y|>|x_\lambda-y|$ for $x\in \Sigma_\lambda$ and
$y\in \Omega_\lambda$. Thus, (i) and (ii) are concluded.

\end{proof}

\begin{lem}
Assume that $u(x)$ satisfies (\ref{int}) and suppose
$\lambda=\bar\lambda$ in the first case; i.e.
$\Sigma^\prime_\lambda$ is internally tangent to the boundary of
$\Omega$ at some point $P_{\bar\lambda}$ not on $T_{\bar\lambda}$,
then $\Sigma_{\bar\lambda}\cup
T_{\bar\lambda}\cup\Sigma^\prime_{\bar\lambda}=\Omega$. \label{le}
\end{lem}

\begin{proof}

When $N\geq \alpha$, thanks to Lemma \ref{mono},
$u_{\bar\lambda}(x)\geq u(x)$ for $x\in \Sigma_{\bar\lambda}.$ While
$ N<\alpha$, $u_{\bar\lambda}(x)\leq u(x)$ for $x\in
\Sigma_{\bar\lambda}.$ We argue by contradiction. Suppose $
\Sigma_{\bar\lambda}\cup
T_{\bar\lambda}\cup\Sigma^\prime_{\bar\lambda}\varsubsetneqq\Omega;$
that is, $\Omega_{\bar\lambda}\not=\emptyset.$ At $
P_{\bar\lambda}$, from (\ref{Com}) and (\ref{Com1}),
$u(P_{\bar\lambda})>u(P)$ in the case of $ N\geq \alpha$. It is a
contradiction since $P_{\bar\lambda}, P\in
\partial\Omega$ and $u(P_{\bar\lambda})=u(P)=$ constant. From the
same reason, $u(P_{\bar\lambda})<u(P)$ when $N<\alpha$. It also
contradicts the fact that $u$ is constant on the boundary.
Therefore, the lemma is completed.

\end{proof}

\begin{lem}
Assume that $u(x)$ satisfies (\ref{int}) and suppose that the second
case occurs: i.e. $T_{\bar\lambda}$ reaches a position where is
orthogonal to the boundary of $\Omega$ at some point $Q$, then,
$\Sigma_{\bar\lambda}\cup
T_{\bar\lambda}\cup\Sigma^\prime_{\bar\lambda}=\Omega$. \label{la}
\end{lem}
\begin{proof}

 Since $u(x)$ is constant on the boundary and $\Omega\in C^1$,
 $\grad u$ is parallel to the normal at $Q$. As implied in the
 second case, $\frac{\partial u}{\partial x_1}|_Q=0$. We denote the coordinate of $Q$ by $z$. Suppose
 $\Omega_{\bar\lambda}\not=\emptyset$, there exits a ball $
 B\subset\subset
 \Omega_{\bar\lambda}$.
 Choose a sequence $\{x^i\}^\infty_1\in\Sigma_{\bar\lambda}\setminus T_{\bar\lambda}$ such that
$x^i\to z$ as $i\to \infty$. It is easy to see that
$x^i_{\bar\lambda}\to z$ as $i\to \infty$. Since $ B\subset\subset
 \Omega_{\bar\lambda},$ we can also find a $ \delta$ such that $ diam
 \Omega>|x^i_{\bar\lambda}-y|>\delta$ for any $y\in B$ and any $x^i_{\bar\lambda}$.

  If
 $N=\alpha$, by (\ref{Com}),
 $$u(x^i_{\bar\lambda})-u(x^i)=\int_{\Omega_{\bar\lambda}}\log\frac{|x^i-y|}{|x^i_{\bar\lambda}-y|}
 \,dy.$$
 Let $e_1=(1,0,\cdots,0)\in \mathbb R^N$, then $(x^i_{\bar\lambda}-x^i)\cdot e_1$ is the first component of
  $(x^i_{\bar\lambda}-x^i)$.
  By the Mean Value theorem,
 \begin{eqnarray}
 \frac{u(x^i_{\bar\lambda})-u(x)}{(x^i_{\bar\lambda}-x^i)\cdot e_1} &=&
 \int_{\Omega_{\bar\lambda}}\frac{\log|x^i-y|-\log|x^i_{\bar\lambda}-y|}{(x^i_{\bar\lambda}-x^i)\cdot e_1}\,dy \nonumber \\
& = & \int_{\Omega_{\bar\lambda}}\frac{(y-\bar
x^i_{\bar\lambda})\cdot e_1}{|y-\bar x^i_{\bar\lambda}|^2}\,dy \nonumber \\
 & > & C\int_B \frac{1}{|diam\,\,\Omega|^2}\,dy  \nonumber \\
 &>& C, \nonumber \\
 \label{Co}
 \end{eqnarray}
 where $\bar
x^i_{\bar\lambda}$ is some point between $ x^i_{\bar\lambda}$ and
$x^i$. Nevertheless,
$$\lim_{i\to
\infty}\frac{u(x^i_{\bar\lambda})-u(x^i)}{(x^i_{\bar\lambda}-x^i)\cdot
e_1}=\frac{\partial u}{\partial x_1}|_Q=0,$$ which contradicts
(\ref{Co}). Therefore, $\Omega_{\bar\lambda}=\emptyset.$

In the case of $N>\alpha$, similarly we have

\begin{eqnarray}
\frac{u(x^i_{\bar\lambda})-u(x^i)}{(x^i_{\bar\lambda}-x^i)\cdot
e_1}&=&\int_{\Omega_{\bar\lambda}}\frac{|x^i_{\bar\lambda}-y|^{\alpha-N}-|x^i-y|^{\alpha-N}}
{(x^i_{\bar\lambda}-x^i)\cdot e_1} \,dy \nonumber \\
&=&\int_{\Omega_{\bar\lambda}}(\alpha-N)|\bar
x^i_{\bar\lambda}-y|^{\alpha-N-2}((x^i_{\bar\lambda}-y)\cdot
e_1)\,dy
\nonumber \\
&>& \int_{B}(\alpha-N)|\bar
x^i_{\bar\lambda}-y|^{\alpha-N-2}((x^i_{\bar\lambda}-y)\cdot e_1)\,dy \nonumber \\
 &>& C.
\end{eqnarray}
It also contradicts  $\frac{\partial u}{\partial x_1}|_Q=0$, thus
$\Omega_{\bar\lambda}=\emptyset$.

 The same idea can be applied to
the case of $N<\alpha$ with minor modification. In conclusion,
$\Sigma_{\bar\lambda}\cup
T_{\bar\lambda}\cup\Sigma^\prime_{\bar\lambda}=\Omega$ when the
second case occurs.
\end{proof}

Combining Lemma (\ref{le}) and Lemma (\ref{la}), Theorem \ref{1t} is
implied.

\section{Proof of Theorem \ref{2t}}

In this section, we will prove theorem \ref{2t} under some
restriction on the diameter of
 $\Omega$. Since we are mainly interested in
 the case of
$\frac{\alpha-N}{2}\in \mathbb N_0$. This is the case when the
fundamental solution of $(-\bigtriangleup)^{\frac{\alpha}{2}}$ has
the representation (\ref{green}). Therefore, we will assume
$\alpha>N$ in this section. Obviously, $u \in C^{1}(\mathbb R^N)$ in
(\ref{log}). We begin with establishing  several lemmas.

\begin{lem}

For $ \lambda_0<\lambda<\bar\lambda$, assume $u(x)$ satisfies
(\ref{log}) with $diam\, \Omega< e^{\frac{1}{N-\alpha}}$, then
$u_\lambda(x)<u(x)$ for any $x\in \Sigma_\lambda$.
\end{lem}

\begin{proof}

Since $|x_\lambda-y_\lambda|=|x-y|,$ and
$|x_\lambda-y|=|x-y_\lambda|,$ we write $u(x)$ and $u_\lambda(x)$ in
the following forms:
\begin{eqnarray*}
u(x)&=&
\int_{\Sigma_\lambda}|x-y|^{\alpha-N}\log\frac{1}{|x-y|}\,dy+\int_{\Sigma_\lambda}|x_\lambda-y|^{\alpha-N}\log\frac{1}{|x_\lambda-y|}\,dy
\\ &&{}+ \int_{\Omega_\lambda}|x-y|^{\alpha-N}\log\frac{1}{|x-y|}\,dy,
\end{eqnarray*}
and
\begin{eqnarray*}
u_\lambda(x)&=&\int_{\Sigma_\lambda}|x_\lambda-y|^{\alpha-N}\log\frac{1}{|x_\lambda-y|}\,dy+\int_{\Sigma_\lambda}|x-y|^{\alpha-N}\log\frac{1}{|x-y|}\,dy
\\ &&{}+
\int_{\Omega_\lambda}|x_\lambda-y|^{\alpha-N}\log\frac{1}{|x_\lambda-y|}\,dy.
\end{eqnarray*}
Then,
\begin{equation}
u_\lambda(x)-u(x)=\int_{\Omega_\lambda}|x-y|^{\alpha-N}\log|x-y|\,dy-\int_{\Omega_\lambda}|x_\lambda-y|^{\alpha-N}\log|x_\lambda-y|\,dy.
\label{pw}
\end{equation}
 We consider the function $s^{\alpha-N}\log s.$ Note $\alpha>N$, thus
$$(s^{\alpha-N}\log s)^\prime=s^{\alpha-N-1}[(\alpha-N)\log s+1]<0,$$
whenever $s<e^{\frac{1}{N-\alpha}}.$ Since $|x-y|>|x_\lambda-y|$ for
$x\in \Sigma_\lambda , \, y\in \Omega_\lambda$, and  $diam \Omega<
e^{\frac{1}{N-\alpha}}$, we easily infer that $u_\lambda(x)<u(x)$
for any $x\in\Sigma_\lambda.$

\end{proof}

\begin{lem}
$u(x)$ satisfies (\ref{log}) and suppose $\lambda=\bar\lambda$ in
the first case; i.e. $\Sigma^\prime_{\bar\lambda}$ is internally
tangent to the boundary of $\Omega$ at some point $P_{\bar\lambda}$
not on $T_{\bar\lambda}$, then $\Sigma_{\bar\lambda}\cup
T_{\bar\lambda}\cup\Sigma^\prime_{\bar\lambda}=\Omega$.
\end{lem}
\begin{proof}

The proof is essentially the same as that of Lemma (\ref{le}).

\end{proof}

\begin{lem}
Suppose that $u(x)$ satisfies (\ref{log}) with $diam \Omega<
e^{\frac{1}{N-\alpha}}$ and  that the second case occurs: i.e.
$T_{\bar\lambda}$ reaches a position where is orthogonal to the
boundary of $\Omega$ at some point $Q$, then,
$\Sigma_{\bar\lambda}\cup
T_{\bar\lambda}\cup\Sigma^\prime_{\bar\lambda}=\Omega$.
\end{lem}

\begin{proof}

The argument follows that of the proof of  Lemma (\ref{la}). Since
$u(x)$ is constant on $\partial\Omega$ and $\Omega\in C^1$,
 $\frac{\partial u}{\partial x_1}|_Q=0$. We denote the coordinate of $Q$ by $z$. Suppose
 $\Omega_{\bar\lambda}\not=\emptyset$, there exits a ball $
 B\subset\subset
 \Omega_{\bar\lambda}$.
 Choosing a sequence $\{x^i\}^\infty_1\in\Sigma_{\bar\lambda}\setminus T_{\bar\lambda}$ such that
$x^i\to z$ as $i\to \infty$,  then $x^i_{\bar\lambda}\to z$ as $i\to
\infty$. Since $ B\subset\subset
 \Omega_{\bar\lambda},$ we find a $ \delta$ such that $ diam
 \Omega>|x^i_{\bar\lambda}-y|>\delta$ for any $y\in B$ and any $x^i_{\bar\lambda}$.

From (\ref{pw}), by Mean Value Theorem,
\begin{eqnarray}
\frac{u(x^i_{\bar\lambda})-u(x^i)}{(x^i_{\bar\lambda}-x^i)\cdot
e_1}&=&\int_{\Omega_{\bar\lambda}}\frac{|x^i-y|^{\alpha-N}\log|x^i-y|
-|x^i_{\bar\lambda}-y|^{\alpha-N}\log|x^i_{\bar\lambda}-y|}
{(x^i_{\bar\lambda}-x^i)\cdot e_1} \,dy \nonumber \\
&=&\int_{\Omega_{\bar\lambda}} -|\bar
x^i_{\bar\lambda}-y|^{\alpha-N-2}((x^i_{\bar\lambda}-y)\cdot
e_1)((\alpha-N)\log|\bar x^i_{\bar\lambda}-y|+1) \,dy
\nonumber \\
&<&\int_{B} -|\bar
x^i_{\bar\lambda}-y|^{\alpha-N-2}((x^i_{\bar\lambda}-y)\cdot
e_1)((\alpha-N)\log|\bar x^i_{\bar\lambda}-y|+1) \,dy
\nonumber \\
 &<& -C. \nonumber \\
 \label{22}
\end{eqnarray}
Where $\bar x^i_{\bar\lambda}$ is some point between $
x^i_{\bar\lambda}$ and $x^i$. The assumption $diam \Omega<
e^{\frac{1}{N-\alpha}}$ is applied in the last inequalities.
Consequently, (\ref{22}) contradicts  $\frac{\partial u}{\partial
x_1}|_Q=0$ as $i\to\infty$. Therefore, the lemma is verified.
\end{proof}

With the help of the above two lemmas, Theorem \ref{2t} is
confirmed.

\end{document}